\documentclass[11pt]{article}

\usepackage[T1]{fontenc}
\usepackage{amsmath,amssymb,amsthm,mathtools}
\usepackage{libertinus}
\usepackage[a4paper,margin=30mm]{geometry}
\usepackage{microtype}
\usepackage{enumitem}
\usepackage{xcolor}
\usepackage[
  colorlinks=true,
  linkcolor=blue!50!black,
  citecolor=blue!50!black,
  urlcolor=blue!50!black,
  pdftitle={A Proof of the Isaacs--Navarro--Wolf Conjecture},
  pdfauthor={Quanfu Yan and Jiping Zhang}
]{hyperref}

\setlist{nosep}
\newtheorem{theorem}{Theorem}[section]
\newtheorem{lemma}[theorem]{Lemma}
\newtheorem{proposition}[theorem]{Proposition}

\theoremstyle{remark}

\newcommand{\Irr}{\operatorname{Irr}}
\newcommand{\Fitt}{\operatorname{F}}
\newcommand{\rad}{\operatorname{rad}}
\newcommand{\Ftwo}{\mathbf F_2}
\newcommand{\normal}{\mathrel{\triangleleft}}

\title{A Proof of the Isaacs--Navarro--Wolf Conjecture}
\author{Quanfu Yan \and Jiping Zhang}
\date{}

\begin{document}

\maketitle

\begin{abstract}
An element of a finite group is non-vanishing if no irreducible complex character vanishes on it.  We prove that every non-vanishing element of a finite solvable group belongs to the Fitting subgroup.  The proof combines the reduction of Isaacs, Navarro, and Wolf with an argument using finite symplectic spaces and Clifford theory.
\end{abstract}



\section{Introduction}

Let $G$ be a finite group.  An element $x\in G$ is called \emph{non-vanishing} if $\chi(x)\ne0$ for every $\chi\in\Irr(G)$.  Isaacs, Navarro, and Wolf introduced these elements and conjectured the assertion in the following theorem \cite{INW1999}.

\begin{theorem}
\label{thm:main}
If $G$ is solvable and $x$ is non-vanishing in $G$, then $x\in\Fitt(G)$.
\end{theorem}

The original paper \cite{INW1999} proves the assertion when $x$ has odd order and, more generally, when the Sylow $2$-subgroups of $G$ are abelian.  For arbitrary finite groups, Dolfi, Navarro, Pacifici, Sanus, and Tiep \cite{DolfiEtAl2010} proved that a non-vanishing element $x$ belongs to $\Fitt(G)$ whenever $(|x|,6)=1.$  For solvable groups, let $F_i(G)$ denote the $i$th term of the ascending Fitting series.  Moret\'o and Wolf showed that every non-vanishing element belongs to $F_{10}(G)$, and Yang improved this bound to $F_8(G)$ \cite{MoretoWolf2004,Yang2009}.  Wolf  \cite{Wolf2020} later proved the conjecture when the irreducible constituents of $\Fitt(G)/\Phi(G)$ are primitive modules. Another special case appears in \cite{AhanjidehRobati2023}.  The conjecture nevertheless remained open in general.

We briefly describe the main ideas of the proof.  Lemma \ref{lem:component-reduction} refines Wolf's quasi-primitive reduction so that all relevant components are retained simultaneously. Lemma~\ref{lem:commutator-data} places Wolf's three cases \cite{Wolf2014} in a common symplectic setting, in which commutators are described by an alternating bilinear form over $\Ftwo$. The symplectic argument in Section~4 uses the hypothesis in Theorem~\ref{thm:linear} that every vector is fixed by some conjugate of $h$ to find two vectors on which this alternating bilinear form has value $1$.  A Clifford-theoretic argument then gives an irreducible character that vanishes at $h$.

All groups in the paper are finite, and all ordinary characters are complex characters.  We use the notation of \cite{Isaacs1976}.  Thus $x^g=g^{-1}xg$, $G'=[G,G]$, and $\Fitt(G)$ is the Fitting subgroup of $G$.  If $V$ is a right $G$-module, then $C_V(x)=\{v\in V:v^x=v\}$.

\section{The quasi-primitive reduction}

\begin{proposition}
\label{prop:linear-reduction}
Theorem~\ref{thm:main} follows from Theorem~\ref{thm:linear} below.
\end{proposition}

The reduction uses the following result on faithful irreducible modules.

\begin{theorem}
\label{thm:linear}
Let $H$ be solvable, let $V$ be a faithful irreducible $H$-module of odd characteristic, and suppose that $1\ne h\in\Fitt(H)$ is an involution satisfying
\begin{equation}
  V=\bigcup_{g\in H}C_V(h^g).
  \label{eq:cover}
\end{equation}
Then $\psi(h)=0$ for some $\psi\in\Irr(H)$.
\end{theorem}

\begin{proof}[Proof of Proposition~\ref{prop:linear-reduction}]
Suppose that Theorem~\ref{thm:main} is false, and choose a counterexample $G$ of minimal order with a non-vanishing element $x\notin\Fitt(G)$.  If $1\ne N\normal G$, then $xN$ is non-vanishing in $G/N$, and hence $xN\in\Fitt(G/N)$.  If $N_1$ and $N_2$ were distinct minimal normal subgroups, the inverse images of $\Fitt(G/N_i)$ would have a nilpotent intersection containing $x$, because $N_1\cap N_2=1$.  This would put $x$ in $\Fitt(G)$.  Thus $G$ has a unique minimal normal subgroup, say $A$.

Minimality also gives $\Phi(G)=1$.  Gasch\"utz's theorem \cite[Theorem~1.12]{ManzWolf1993} now yields
\[
  \Fitt(G)=A,\qquad G=A\rtimes H,
\]
where $A$ is an elementary abelian $p$-group and $H\cong G/A$ acts faithfully and irreducibly on $V=\Irr(A)$.  Put $h=xA$.  Inflation from $H\cong G/A$ shows that $h$ is non-vanishing in $H$.  Since $h\ne1$, minimality gives $h\in\Fitt(H)$.

By \cite[Lemma~2.3]{INW1999}, the element $x$ fixes a member of every $G$-orbit on $\Irr(A)$.  The actions of $A$ on $\Irr(A)$ and of $xA$ on $\Irr(A)$ are respectively trivial and equal to the action of $x$. Consequently $h$ fixes a member of every $H$-orbit on $V$, which is equivalent to \eqref{eq:cover}.  Theorem~4.2 of \cite{INW1999} gives $h^2=1$.

Next we show that $p\ne2$.  Otherwise $V$ has characteristic $2$.  A finite $2$-group acting on a nonzero vector space of characteristic $2$ has a nonzero fixed point, so $C_V(O_2(H))\ne0.$ Since $O_2(H)\normal H$, this fixed-point space is $H$-invariant, and hence equals $V$ by irreducibility.  Faithfulness gives $O_2(H)=1$.  On the other hand, the involution $h\in\Fitt(H)$ lies in $O_2(H)$, a contradiction. Theorem~\ref{thm:linear} would now give an irreducible character of $H$ vanishing at $h$, contrary to the non-vanishing of $h$.
\end{proof}

We next record a simultaneous form of the imprimitive reduction needed for Theorem~\ref{thm:linear}.  Wolf's Corollary~2.3 \cite{Wolf2014} follows one selected nonidentity component of a single element; the lemma below simultaneously follows all nonidentity components of a finite set. A $K$-module $U$ is \emph{quasi-primitive} if $U_N$ is homogeneous for every $N\normal K$.  A system of imprimitivity for $U$ is a decomposition of its underlying vector space $U=U_1\oplus\cdots\oplus U_m$ with $m>1$, where $U_i$ are nonzero subspaces which are permuted transitively by $K$.  The subspaces $U_i$ need not be $K$-submodules.

\begin{lemma}
\label{lem:component-reduction}
Let $K$ act faithfully and irreducibly on $U$, and let $\mathcal X\subseteq\Fitt(K)\setminus\{1\}$ be a finite set of involutions. Suppose that, for every $a\in\mathcal X$,
\[
  U=\bigcup_{g\in K}C_U(a^g).
\]
Then there are a faithful irreducible quasi-primitive $J$-module $W$, a transitive permutation group $S$ of degree $n$, and embeddings
\[
  K\hookrightarrow J\wr S=J^n\rtimes S,
  \qquad U\cong W^n,
\]
such that every $a\in\mathcal X$ belongs to the base group.  If $a=(a_1,\ldots,a_n)$, then every $a_i\ne1$ is an involution in $\Fitt(J)$ and
\[
  W=\bigcup_{j\in J}C_W(a_i^j).
\]
\end{lemma}

\begin{proof}
We use induction on $\dim U$.  If $U$ is quasi-primitive, take $J=K$, $W=U$, $n=1$, and $S=1$.  Otherwise modular Clifford theory \cite[Theorem~0.1]{ManzWolf1993} supplies a nontrivial system of imprimitivity.  Choose one with as few components as possible,
\[
  U=U_1\oplus\cdots\oplus U_m.
\]
The permutation group induced by $K$ on the $U_i$ is primitive since a nontrivial block system would give a $K$-invariant decomposition with fewer components.  Hence $K_i=N_K(U_i)$ is maximal in $K$.  If $0<U_i'<U_i$ were $K_i$-invariant, the direct sum of its distinct $K$-conjugates would be a nonzero proper $K$-submodule of $U$; thus $U_i$ is an irreducible $K_i$-module.

Fix $a\in\mathcal X$ and set $L=\langle a^K\rangle$.  Then $L\normal K$ is nilpotent.  If $a\notin K_i$, put $M=L\cap K_i$.  The normalizer condition for finite nilpotent groups gives
\[
  M<N_L(M).
\]
Since $L\normal K$, we have $M\normal K_i$.  An element of $N_L(M)\setminus M$ lies outside $K_i$, so $N_K(M)$ strictly contains the maximal subgroup $K_i$.  Hence $M\normal K$.  By assumption, a nonzero vector in $U_i$ is fixed by some conjugate $a^g$.  An element which moves $U_i$ to a different direct summand cannot fix such a vector, and therefore $a^g\in L\cap K_i=M$. Normality of $M$ gives $a\in M$, a contradiction.  Thus every member of $\mathcal X$ lies in the kernel of the permutation action on the components.  This kernel is normal in $K$ and contains $a$, so it contains $L=\langle a^K\rangle$ and $L\le\bigcap_iK_i$.

Let $C_i=C_{K_i}(U_i)$ and $\widehat K_i=K_i/C_i$.  The usual imprimitivity embedding \cite[Lemma~2.8]{ManzWolf1993} gives
\[
  K\hookrightarrow\widehat K_i\wr S_0,
  \qquad U\cong U_i^m,
\]
where $S_0$ is the permutation group induced on the components.  By definition, $\widehat K_i$ acts faithfully on $U_i$, and this action is irreducible by the preceding paragraph.  The image of every nonidentity component of $a$ belongs to $\Fitt(\widehat K_i)$. Indeed, if $B=\bigcap_iK_i$ and $\rho_i:K_i\to\widehat K_i$ is the quotient map, then the $i$th component of $a$ lies in $\rho_i(L)$.  This is a normal nilpotent subgroup of $\widehat K_i$, because $L\normal K$. For $w\in U_i$, by the hypothesis  $U=\bigcup_{g\in K}C_U(a^g),$ some conjugate of $a$ fixes the diagonal vector $\Delta(w)=(w,\ldots,w)$.  Since $a$ has trivial permutation part, wreath-product conjugation shows that each nonidentity component of $a$ has a conjugate in $\widehat K_i$ fixing $w$.  Thus, if $c$ is any nonidentity component of $a$, then
\[
  U_i=\bigcup_{k\in\widehat K_i}C_{U_i}(c^k).
\]
Here the element $k$ may depend on both $c$ and $w$.

Let $\mathcal X_1$ be the set of all these nonidentity components, as $a$ ranges over $\mathcal X$. Since $a$ is an involution in the base group, each of its components is an involution.  By the preceding two paragraphs, $\mathcal X_1$ is a finite set of involutions in $\Fitt(\widehat K_i)\setminus\{1\}$ and satisfies the required fixed-space condition.  Thus $(\widehat K_i,U_i,\mathcal X_1)$ satisfies all the inductive hypotheses and $\dim U_i<\dim U$.  By induction there are a faithful irreducible quasi-primitive $J$-module $W$, a transitive permutation group $S_1$ of
degree $r$, and embeddings
\[
  \widehat K_i\hookrightarrow J\wr S_1,
  \qquad U_i\cong W^r,
\]
with the asserted conclusion for every member of $\mathcal X_1$.  Apply this embedding in each of the $m$ coordinates of $\widehat K_i\wr S_0$.  By the associativity of wreath products in their permutation actions \cite[Theorem~7.26]{Rotman1995},
\[
  K\hookrightarrow (J\wr S_1)\wr S_0
   \cong J\wr(S_1\wr S_0),
  \qquad U\cong (W^r)^m\cong W^{rm}.
\]
In its imprimitive action, $S_1\wr S_0$ is transitive of degree $rm$. Since every $a\in\mathcal X$ has trivial permutation part and its nonidentity components belong to $\mathcal X_1$, the inductive conclusion applies to every nonidentity component in the final base group.  This proves the lemma.
\end{proof}

Apply Lemma~\ref{lem:component-reduction} to $(H,V,\{h\})$.  We obtain
\begin{equation}
  H\hookrightarrow J\wr S=J^n\rtimes S,
  \qquad V\cong W^n,
  \label{eq:wreath}
\end{equation}
where $W$ is a faithful irreducible quasi-primitive $J$-module.  Write $h=(c_1,\ldots,c_n)$.  Every $c_i\ne1$ is an involution in $\Fitt(J)$ and its conjugate fixed spaces cover $W$.

Wolf's theorem \cite[Theorem~2.1]{Wolf2014} now leaves exactly three possibilities:
\begin{enumerate}[label=(\roman*)]
\item $|W|=q^2$, where $q$ is a Mersenne prime, and $J=\Fitt(J)=SD\times C\leq\Gamma(W)$, where $C$ is cyclic of odd order and $SD$ is semidihedral of order $4(q+1)$;
\item $|W|=25$, $\Fitt(J)\cong Q_8\circ C_4$, and $J/\Fitt(J)\cong C_3$ or $S_3$;
\item $|W|=81$, $\Fitt(J)\cong Q_8\circ D_8$, and, since $J$ is solvable, $J/\Fitt(J)$ is isomorphic to $C_5$, $D_{10}$, or the Frobenius group of order $20$.
\end{enumerate}
Here $\circ$ denotes the central product obtained by identifying the central involutions of the two factors, and $D_8$ denotes the dihedral group of order $8$.  Moreover, for every
$0\ne w\in W$,
\begin{equation}
  C_{\Fitt(J)}(w)=\langle t_w\rangle\cong C_2.
  \label{eq:unique-fixer}
\end{equation}

The next lemma identifies the full $J$-conjugacy class of each nonidentity component of $h$; this will allow us to control all components of the elements in $h^H$.

\begin{lemma}
\label{lem:component-class}
Let $t=c_i$ be a nonidentity component of $h$.  Then
\[
  t^J=\{t_w:0\ne w\in W\}.
\]
\end{lemma}

\begin{proof}
For $0\ne w\in W$, Lemma~\ref{lem:component-reduction} gives $j\in J$ such that $t^j$ fixes $w$.  By \eqref{eq:unique-fixer}, $t^j=t_w$; hence $\{t_w:0\ne w\in W\}\subseteq t^J$.  Conversely, fix $c=t_{w_0}\in t^J$.  If $s=c^k\in t^J$, then $s$ fixes $w_0^k$, so \eqref{eq:unique-fixer} gives $s=t_{w_0^k}$.  This proves the reverse inclusion.
\end{proof}

\section{The three cases and their commutator forms}

We begin by defining the normal subgroup generated by the conjugates of $h$ and recording the information about its components needed in all three cases.  Put
\[
  P_0=\langle h^H\rangle.
\]
Then $P_0$ is normal in $H$.  Also $h\in O_2(H)$, and the normality of $O_2(H)$ gives $h^H\subseteq O_2(H)$.  Hence $P_0\le O_2(H)$.

Let $B_0=J^n$ be the base group in \eqref{eq:wreath}; it is normal in $J\wr S$.  By Lemma~\ref{lem:component-reduction}, $h\in B_0$, so $h^H\subseteq B_0$. By Lemma~\ref{lem:component-class}, every nonidentity component $c_i$ of $h$ belongs to $\{t_w:0\ne w\in W\}$, and this set is $J$-invariant.  If $y=h^g\in h^H$, wreath-product conjugation permutes the components of $h$ and conjugates each component by an element of $J$.  Thus every component of $y$ is either $1$ or some $t_w$.  Since $t_w\in\Fitt(J)$ by \eqref{eq:unique-fixer}, it follows in particular that
\[
  P_0\le \Fitt(J)^n.
\]

Before treating the three cases, we recall the terminology for bilinear and symplectic forms used below, following \cite[Chapters~2 and~3]{Grove2002}.  A bilinear form on an $\Ftwo$-space $E$ is a map $B:E\times E\to\Ftwo$ which is linear in each variable.  It is \emph{alternating} if $B(v,v)=0$ for every $v\in E$. Its radical is
\[
  \rad(E)=\{v\in E:B(v,w)=0\text{ for every }w\in E\}.
\]
The form is nondegenerate if $\rad(E)=0$; a space with a nondegenerate alternating form is called a \emph{symplectic space}.  In each component group $X$ used below, $X/Z(X)$ is an elementary abelian $2$-group and $X'=\langle z\rangle$ has order $2$.  Then
\begin{equation}
  [a,b]=z^{B_X(aZ(X),bZ(X))}
  \label{eq:comm-form-definition}
\end{equation}
defines an alternating bilinear form on $X/Z(X)$.  If $xZ(X)$ belongs to its radical, then $x$ commutes with every element of $X$, so $x\in Z(X)$; hence $B_X$ is nondegenerate.

We first consider the cases $|W|=25$ and $|W|=81$.

\begin{lemma}
\label{lem:central-products}
Assume the case $|W|=25$ or $81$, write $F=\Fitt(J)$ and $E=F/Z(F)$, and let $B$ be the form in \eqref{eq:comm-form-definition}.  Then $B$ is nondegenerate.  The images in $E$ of the involutions $t_w$ form a set $T\subseteq E\setminus\{0\}$ which spans $E$, and every member of $T$ is the image of $t_w$ for some $0\ne w\in W$.
\end{lemma}

\begin{proof}
The form $B$ is nondegenerate by the observation preceding the lemma. First suppose that $|W|=25$.  Thus $F=Q_8\circ C_4$.  Choose generators $a,b$ for $Q_8$ and a central element $c$ of order $4$ so that $a^2=b^2=c^2=z$ and $[a,b]=z.$ Then $Z(F)=\langle c\rangle\cong C_4$, $F'=\langle z\rangle$, and
\[
  E=F/Z(F)=\langle aZ(F),bZ(F)\rangle_{\Ftwo}\cong\Ftwo^2.
\]

We next determine the noncentral involutions of $F$.  The three nonzero cosets of $Z(F)$ are represented by $a$, $b$, and $ab$.  Since $[a,b]=z$, we have $(ab)^2=z$, and thus $x^2=z$ for $x\in\{a,b,ab\}$.  Since $c$ is central and $c^2=z$, for every integer $k$ with $0\le k\le3$ we have $(xc^k)^2=x^2c^{2k}=z^{k+1}.$ It follows that $xc^k$ is an involution exactly when $k$ is odd.  Each nonzero coset of $Z(F)$ therefore contains exactly two involutions, and
$F$ has exactly six noncentral involutions.

Since $F'=\langle z\rangle$ is characteristic in $F=\Fitt(J)$, it is normal in $J$, and so $z\in Z(J)$.  On the irreducible $J$-module $W$, the element $z$ acts as $I$ or $-I$, and faithfulness excludes the first possibility.  Thus $z$ acts as $-I$.

Let $d$ be a noncentral involution of $F$.  Since $dZ(F)\ne0$ and the commutator form is nondegenerate, there is $r_0\in F$ such that $[d,r_0]=z$.  With the convention $x^g=g^{-1}xg$, this gives
\[
  d^{r_0}=d[d,r_0]=dz.
\]
Thus $d$ and $dz$ are conjugate.  Since $z$ acts as $-I$, the fixed space of $dz$ is the $-1$-eigenspace of $d$, and hence is complementary to $C_W(d)$.  Note that conjugate elements have similar representing matrices, so $d$ and $dz$ have fixed spaces of the same dimension.  Since $\dim_{\mathbf F_5}W=2$, both fixed spaces are lines.

Fix a nonidentity component $t$ of $h$ and put $\mathcal C=t^J$.  By Lemma~\ref{lem:component-class},  $ \mathcal C=\{t_w:0\ne w\in W\}.$ The sets $C_W(u)\setminus\{0\}$, for $u\in\mathcal C$, partition $W\setminus\{0\}$.  Indeed, if $0\ne w\in W$, then $t_w\in\mathcal C$ and $w\in C_W(t_w)$, so these sets cover $W\setminus\{0\}$.  If the same nonzero vector $w$ is fixed by $u,v\in\mathcal C$, then $u,v\in C_F(w)=\langle t_w\rangle.$ Since $u$ and $v$ are nonidentity involutions, both equal $t_w$, and hence the sets are pairwise disjoint.  Every $u\in\mathcal C$ fixes a nonzero vector, whereas $z$ does not.  Since $z$ is the unique central involution of $F$, every $u\in\mathcal C$ is noncentral.  By the preceding paragraph, $C_W(u)$ is a line.  Therefore
\[
  |\mathcal C|=\frac{5^2-1}{5-1}=6.
\]
Thus $\mathcal C$ consists of all six noncentral involutions of $F$. Their images in $E$ are the three nonzero vectors of $E$, so these images span $E$.

Now suppose that $|W|=81$.  Here $F=Q_8\circ D_8, Z(F)=F'=\langle z\rangle,$ and $|z|=2.$ Choose generators $a,b$ for $Q_8$ and $r,s$ for $D_8$ so that the two factors commute and
\[
  a^2=b^2=r^2=z,\qquad s^2=1,\qquad [a,b]=[r,s]=z.
\]
Consequently
\[
  E=F/Z(F)
   =\langle\bar a,\bar b,\bar r,\bar s\rangle_{\Ftwo}
   \cong\Ftwo^4,
\]
where bars denote images modulo $Z(F)$.

The element $a^xb^yr^us^v$, where $x,y,u,v\in\Ftwo$, maps to the vector $(x,y,u,v)$ in this basis.  Using
$(gh)^2=g^2h^2[g,h]^{-1}$ when $[g,h]$ is central, and using the fact that the two factors commute, we obtain
\[
  (a^xb^yr^us^v)^2=z^{x+y+xy+u+uv}.
\]
Reading the exponent modulo $2$, the nonzero vectors represented by elements of square $1$ are exactly
\[
  (0,0,0,1),\ (0,0,1,1),\ (0,1,1,0),\
  (1,0,1,0),\ (1,1,1,0).
\]

For any $g\in F$, the two elements mapping to $gZ(F)$ are $g$ and $gz$. Since $z$ is central of order $2$, $(gz)^2=g^2$.  Each of the five vectors just listed therefore corresponds to two noncentral involutions, whereas the two elements mapping to any other nonzero vector have square $z$ and order $4$.  Hence $F$ has exactly ten noncentral involutions, occurring in five pairs modulo $Z(F)$.

As above, $z\in Z(J)$ and acts as $-I$ on $W$. Nondegeneracy shows that each noncentral involution $d$ is conjugate to $dz$.  Their fixed spaces are complementary and have equal dimension, and hence each has dimension $2$ because $\dim_{\mathbf F_3}W=4$.

Put $\mathcal C=t^J$.  By the same argument as in the case $|W|=25$, the sets $C_W(u)\setminus\{0\}$, for $u\in\mathcal C$, partition $W\setminus\{0\}$.  Hence
\[
  |\mathcal C|=\frac{3^4-1}{3^2-1}=10.
\]
The central involution fixes no nonzero vector, so $\mathcal C$ consists of all ten noncentral involutions.  Their images in $E$ are therefore the five vectors displayed above. It is not difficult to see that these vectors span $E.$
\end{proof}

We now consider the case $|W|=q^2$.  Write $q+1=2^m$.  The Sylow $2$-subgroup $SD$ of $J$ is
\[
 SD=SD_{2^{m+2}}=\langle r,t\mid r^{2^{m+1}}=t^2=1,
       \ trt=r^{2^m-1}\rangle.
\]
In the proof of \cite[Theorem~2.1]{Wolf2014}, Wolf shows that the $q+1$ noncentral involutions of the semidihedral Sylow $2$-subgroup form a single conjugacy class.  With the presentation above, these involutions are $r^{2j}t$, and hence
\begin{equation}
  t^J=\{r^{2j}t:j\in\mathbf Z/2^m\mathbf Z\}.
  \label{eq:semidihedral-class}
\end{equation}

\begin{lemma}
\label{lem:semidihedral}
Assume the case $|W|=q^2$, where $q$ is a Mersenne prime.  There is a subgroup $K\normal H$, contained in $P_0$, with the following properties.  In the $i$th copy of $J$, there are subgroups $D_i=\langle u_i,t_i\rangle$ and $A_i=\langle u_i^4\rangle$ such that $D_i/A_i\cong D_8$, and
\[
  \bar h=hK\ne1,
  \qquad
  P_0/K\hookrightarrow
  (D_1/A_1)\times\cdots\times(D_n/A_n)\cong D_8^n.
\]
The embedding commutes with conjugation by $H/K$.  In the $i$th factor put
\[
  E_i=(D_i/A_i)/Z(D_i/A_i),\qquad
  T_i=\{(t_iA_i)Z(D_i/A_i),(u_it_iA_i)Z(D_i/A_i)\}.
\]
Then $E_i$ is a two-dimensional symplectic space, $T_i$ is a basis of $E_i$, and every member of $T_i$ is the image of some $t_w$.
\end{lemma}

\begin{proof}
Write $u=r^2$ and $D=\langle u,t\rangle$.  The defining relation for $SD$ gives $|u|=2^m$ and $tut=u^{-1},$ so $D$ is dihedral of order $2^{m+1}$.  Moreover, if $c=u^jt\in t^J$, then direct calculation gives $c^r c=u^{2^{m-1}-1}.$ The exponent is odd, and hence this element generates $\langle u\rangle$. Thus the normal closure of $c$ contains $u$ and then $t=u^{-j}c$; conversely, all the conjugates in
\eqref{eq:semidihedral-class} lie in $D$.  Therefore $D$ is the normal closure in $J$ of every member of $t^J$.

The subgroup $\langle u\rangle$ is the unique cyclic subgroup of index $2$ in $D$, and is therefore characteristic in $D$.  It follows that $A=\langle u^4\rangle$ is characteristic in $D$.  The element $uA$ has order $4$, so
\[
  D/A=\langle uA,tA\rangle\cong D_8,
  \qquad Z(D/A)=\langle u^2A\rangle.
\]
Modulo this center, the elements $u^jtA$ have one of the two images
\[
  (tA)Z(D/A),\qquad (utA)Z(D/A),
\]
according as $j$ is even or odd.  These are distinct and span $(D/A)/Z(D/A)$.

Use copies $D_i$ and $A_i$ in the $i$th factor of the base group $J^n$, and put
\[
  D_0=D_1\times\cdots\times D_n,\qquad
  A_0=A_1\times\cdots\times A_n,\qquad
  K=P_0\cap A_0.
\]
Write $W^n=W_1\oplus\cdots\oplus W_n$, let $J_i$ be the $i$th coordinate subgroup of $J^n$, and put
\[
  \mathcal C_i=\{t_w\in J_i:0\ne w\in W_i\}.
\]
If $g\in H$ carries $W_i$ onto $W_j$, the map $x\mapsto x^g=g^{-1}xg$ restricts to an isomorphism $J_i\to J_j$.  For $t_w\in\mathcal C_i$, the element $t_w^g$ fixes $w^g\in W_j$ and lies in the Fitting subgroup of $J_j$, so \eqref{eq:unique-fixer} gives $t_w^g=t_{w^g}$.  Thus $\mathcal C_i^g\subseteq\mathcal C_j$; applying the same argument to $g^{-1}$ gives equality.  Since $D_i$ is the normal closure of $\mathcal C_i$ and $A_i$ is characteristic in $D_i$, we have $D_i^g=D_j$ and $A_i^g=A_j$.  Hence $H$ normalizes both $D_0$ and $A_0$.

Every component of every generator in $h^H$ is either $1$ or some $t_w$, so $P_0=\langle h^H\rangle\le D_0$.  Since $P_0\normal H$ and $A_0$ is normalized by $H$, their intersection $K=P_0\cap A_0$ is normal in $H$. Let
\[
  \rho:D_0\longrightarrow
  (D_1/A_1)\times\cdots\times(D_n/A_n)
\]
be the product of the quotient maps.  Its kernel is $A_0$, and hence the kernel of $\rho|_{P_0}$ is $P_0\cap A_0=K$.  The first isomorphism theorem therefore gives the embedding
\[
  \iota:P_0/K\hookrightarrow
  (D_1/A_1)\times\cdots\times(D_n/A_n)\cong D_8^n,
  \qquad \iota(xK)=\rho(x).
\]
The map $\rho$ commutes with conjugation by $H$.  Elements of $K$ have trivial image under $\rho$, so the resulting conjugation action on the image depends only on the coset in $H/K$.  Thus $\iota$ commutes with the conjugation actions of $H/K$.

Since $h\ne1$ and has trivial permutation part, it has a nonidentity component.  Such a component has the form $u_i^jt_i$ and lies outside the rotation subgroup $\langle u_i\rangle$, whereas $A_i\le\langle u_i\rangle$. Thus $h\notin K$ and $hK\ne1$.  Finally, in
\[
  E_i=(D_i/A_i)/Z(D_i/A_i),
\]
the image of $u_i^jt_iA_i$ is $(t_iA_i)Z(D_i/A_i)$ or $(u_it_iA_i)Z(D_i/A_i)$ according as $j$ is even or odd.  These are distinct nonzero vectors in the two-dimensional space $E_i$, and hence form a basis.  The class $t^J$ contains elements $u_i^jt_i$ with $j$ even and elements $u_i^jt_i$ with $j$ odd.  Hence Lemma~\ref{lem:component-class} shows that every member of this basis is the image of some $t_w$.  The commutator form on $E_i$ is nondegenerate and alternating by
\eqref{eq:comm-form-definition}, so $E_i$ is symplectic.
\end{proof}

We now use common notation for the three cases.  In the cases $|W|=25$ and $|W|=81$, set $\bar H=H,$  $\bar h=h$ and $P=P_0.$ In the case $|W|=q^2$, set $\bar H=H/K,$ $\bar h=hK$ and  $P=P_0/K.$ Let $X_i$ denote the $i$th component group used below: $F$ in the cases $|W|=25$ and $|W|=81$, and $D_8$ in the case $|W|=q^2$.  Write $E_i=X_i/Z(X_i),$ $X_i'=\langle z_i\rangle,$ and let $B_i$ be the commutator form on $E_i$.  The direct sum
\[
  \mathcal E=E_1\oplus\cdots\oplus E_n
\]
is equipped with the bilinear form
\[
  B((v_i),(w_i))=\sum_{i=1}^nB_i(v_i,w_i).
\]
Conjugation by $H$ induces an action of $\bar H$ on $\mathcal E$.  If $g\in H$ carries $X_i$ onto $X_j$, this action is given by $(xZ(X_i))^{\bar g}=x^gZ(X_j)$, where $x\in X_i$, $\bar g\in\bar H$, and $g\in H$ is a representative of $\bar g$.  In the case $|W|=q^2$, independence of the representative follows from Lemma~\ref{lem:semidihedral}. Thus different summands are orthogonal, that is,  $B(E_i,E_j)=0$ for $i\ne j$.
The group $\bar H$ acts transitively on the set of summands and preserves $B$, that is, for any $v,w\in\mathcal E$ and $\bar g\in\bar H$, $B(v^{\bar g},w^{\bar g})=B(v,w).$ If an element of $\bar H$ carries $E_i$ onto $E_j$, then it carries $T_i$ onto $T_j$.  By Lemma~\ref{lem:central-products} and Lemma~\ref{lem:semidihedral}, each $T_i$ spans $E_i$, and each of its members is represented by some $t_w$.

Write $p=(p_1,\ldots,p_n)\in P$ in the product of the component groups and define
\[
  \pi:P\longrightarrow\mathcal E,
  \qquad
  \pi(p)=(p_1Z(X_1),\ldots,p_nZ(X_n)).
\]
Put
\[
  \mathcal Y=\pi(\bar h^{\,\bar H}),
  \qquad
  \mathcal U=\langle\mathcal Y\rangle_{\Ftwo}.
\]

The common properties needed below are collected in the following lemma.

\begin{lemma}
\label{lem:commutator-data}
We have $\pi(P)=\mathcal U$ and
\[
  [P,P]\le P\cap Z_0\le Z(P),
  \qquad
  Z_0=\langle z_1\rangle\times\cdots\times\langle z_n\rangle.
\]
The linear character $\beta_0\in\Irr(Z_0)$ defined by
\[
  \beta_0(z_1^{e_1}\cdots z_n^{e_n})=(-1)^{e_1+\cdots+e_n}
\]
is $\bar H$-invariant, and, for all $a,b\in P$,
\begin{equation}
  \beta_0([a,b])=(-1)^{B(\pi(a),\pi(b))}.
  \label{eq:commutator-formula}
\end{equation}
Moreover $\mathcal Y=\pi(\bar h)^{\bar H}$ is a nonzero $\bar H$-orbit.
\end{lemma}

\begin{proof}
Since $P=\langle\bar h^{\bar H}\rangle$, its image under $\pi$ is exactly $\mathcal U$.  Each $X_i$ has derived subgroup $\langle z_i\rangle$ in its center, and commutators in a direct product are computed componentwise. Hence $[P,P]\le P\cap Z_0\le Z(P)$.

The definition of $\beta_0$ gives a homomorphism $Z_0\to\{\pm1\}$ because the coordinates $e_i$ are added modulo $2$.  If $a=(a_i)$ and $b=(b_i)$, then
\[
  [a,b]=\prod_i[a_i,b_i]
       =\prod_i z_i^{B_i(\pi_i(a),\pi_i(b))}.
\]
Applying $\beta_0$ gives \eqref{eq:commutator-formula}.  Conjugation by $\bar H$ permutes the $z_i$ and preserves each commutator form, so it preserves both $\beta_0$ and $B$.  The definition of $\pi$ gives $\pi(x^g)=\pi(x)^g$, whence $\mathcal Y=\pi(\bar h)^{\bar H}$.  Finally $\bar h\ne1$ has a nonidentity component with nonzero image, so $0\notin\mathcal Y$.
\end{proof}

\section{A symplectic consequence of Rado's theorem}

We first recall the vector-space form of Rado's independent-representatives theorem \cite{Rado1942}.  A short proof is included.

\begin{lemma}
\label{lem:independent-representatives}
Let $A_1,\ldots,A_n$ be finite subsets of a vector space $L$.  If
\begin{equation}
  \dim\left\langle\bigcup_{i\in I}A_i\right\rangle\ge |I|
  \label{eq:rado-condition}
\end{equation}
for every $I\subseteq\{1,\ldots,n\}$, then there are $a_i\in A_i$ such that $a_1,\ldots,a_n$ are linearly independent.
\end{lemma}

\begin{proof}
We argue by induction on $n$.  The case $n=1$ follows from \eqref{eq:rado-condition}, since $A_1$ then contains a nonzero vector. Suppose first that equality holds in \eqref{eq:rado-condition} for some nonempty proper set $I$, and put
\[
  M=\left\langle\bigcup_{i\in I}A_i\right\rangle.
\]
By induction, the sets indexed by $I$ have independent representatives $a_i$, and these representatives form a basis of $M$, since $\dim M=|I|$.  Let $q:L\to L/M$ be the quotient map.  If $J\subseteq\{1,\ldots,n\}\setminus I$, then
\[
 \dim\left\langle\bigcup_{j\in J}q(A_j)\right\rangle
 =\dim\left\langle\bigcup_{k\in I\cup J}A_k\right\rangle-\dim M
 \ge |J|.
\]
Induction in $L/M$ therefore gives representatives $a_j\in A_j$ whose images are independent.  Together with the previously chosen $a_i$, they are independent in $L$.

It remains to consider the case in which no nonempty proper set gives equality.  Then every such set $I$ spans a space of dimension at least $|I|+1$.  Choose $0\ne a_n\in A_n$ and let $q:L\to L/\langle a_n\rangle$ be the quotient map.  For every $I\subseteq\{1,\ldots,n-1\}$, passage to this quotient decreases dimension by at most one, and hence
\[
  \dim\left\langle\bigcup_{i\in I}q(A_i)\right\rangle\ge |I|.
\]
Induction gives $a_i\in A_i$, for $i<n$, whose images are independent. It follows immediately that $a_1,\ldots,a_n$ are independent in $L$.
\end{proof}

For $0\ne y=(y_i)\in\mathcal E$, define
\[
  \mathcal C_y=
  \{(x_i)\in\prod_iT_i:y_i\in\langle x_i\rangle
     \text{ for every }i\}.
\]
Since the ground field is $\Ftwo$, this requires $y_i=x_i$ when $y_i\ne0$ and imposes no condition when $y_i=0$.

We now state the result used in the final argument.

\begin{theorem}
\label{thm:symplectic-selection}
At least one of the following holds:
\begin{enumerate}[label=(\roman*)]
\item $\displaystyle\bigcup_{y\in\mathcal Y}\mathcal C_y
      \ne\prod_iT_i$;
\item $\mathcal Y\cap\rad_B(\mathcal U)=\varnothing$, where $\rad_B(\mathcal U)=
  \{u\in\mathcal U:B(u,v)=0\text{ for every }v\in\mathcal U\}.$
\end{enumerate}
\end{theorem}

\begin{proof}
Suppose that some $y\in\mathcal Y$ belongs to $\rad_B(\mathcal U)$.  The set $\mathcal Y$ is one $\bar H$-orbit and both $B$ and $\mathcal U$ are $\bar H$-invariant.  Hence every member of $\mathcal Y$ belongs to the radical.  Since $\mathcal Y$ spans $\mathcal U$, we have $B(u,v)=0$ for all $u,v\in\mathcal U$.

Let $\mathcal E_I=\bigoplus_{i\in I}E_i$ and write $\dim E_i=2m_i$.  A subspace on which a nondegenerate alternating form vanishes has dimension at most half the dimension of the whole space. Applied to $\mathcal U\cap\mathcal E_I$, the dimension bound gives
\[
  \dim(\mathcal U\cap\mathcal E_I)\le\sum_{i\in I}m_i.
\]
Let $q:\mathcal E\to\mathcal E/\mathcal U$ be the quotient map.  Since $T_i$ spans $E_i$,
\[
  \left\langle\bigcup_{i\in I}q(T_i)\right\rangle=q(\mathcal E_I).
\]
The kernel of $q|_{\mathcal E_I}$ is $\mathcal U\cap\mathcal E_I$, and therefore
\[
 \dim\left\langle\bigcup_{i\in I}q(T_i)\right\rangle
 =\dim\mathcal E_I-\dim(\mathcal U\cap\mathcal E_I)
 \ge\sum_{i\in I}m_i\ge|I|.
\]
Lemma~\ref{lem:independent-representatives} gives $x_i\in T_i$ such that $x_i+\mathcal U$ are linearly independent in
$\mathcal E/\mathcal U$.  Equivalently,
\[
  \mathcal U\cap\bigoplus_i\langle x_i\rangle=0.
\]
If $(x_i)\in\mathcal C_y$ for some $y\in\mathcal Y$, then the nonzero vector $y$ belongs to both subspaces, a contradiction.  Thus (i) holds. If no member of $\mathcal Y$ belongs to the radical, then (ii) holds.
\end{proof}

\section{Completion of the proof}

We need one elementary consequence of Clifford theory.

\begin{lemma}
\label{lem:central-commutator-zero}
Let $P\normal L$, let $P'\le Z(P)$, and let $\lambda\in\Irr(P')$ be $L$-invariant.  Let $a\in P$.  Suppose that for every $y\in a^L$ there is $p_y\in P$ such that
\[
  \lambda([y,p_y])\ne1.
\]
Then some $\chi\in\Irr(L)$ satisfies $\chi(a)=0$.
\end{lemma}

\begin{proof}
Let $\theta\in\Irr(P\mid\lambda).$  Frobenius reciprocity gives $[\theta_{P'},\lambda]>0$.  Since $P'\le Z(P)\le Z(\theta)$, Lemma~2.27(c) of \cite{Isaacs1976} gives
$\theta_{P'}=\theta(1)\nu$ for some linear character $\nu$ of $P'$; the positive inner product forces $\nu=\lambda$.

For $y\in a^L$, choose $p=p_y$.  With $[y,p]=y^{-1}y^p$, we have $y^p=y[y,p]$.  Since $\theta_{P'}=\theta(1)\lambda$, if $\rho$ is a representation affording $\theta$, then $\rho([y,p])=\lambda([y,p])I$.  Since $\theta$ is constant on $P$-conjugacy classes,
\[
  \theta(y)=\theta(y^p)
  =\operatorname{tr}\bigl(\rho(y)\rho([y,p])\bigr)
  =\lambda([y,p])\theta(y).
\]
Thus $\theta(y)=0$ for every $y\in a^L$.  Let $\chi\in\Irr(L\mid\theta)$, and choose $g_1,\ldots,g_t\in L$ so that $\theta^{g_1},\ldots,\theta^{g_t}$ are the distinct $L$-conjugates of $\theta$.  Clifford's theorem gives
\[
  \chi_P=e\sum_{i=1}^t\theta^{g_i}
\]
for some positive integer $e$.  Since $\theta^{g_i}(a)=\theta(a^{g_i^{-1}})=0$ for every $i$, it follows that $\chi(a)=0$.
\end{proof}

\begin{proof}[Proof of Theorem~\ref{thm:linear}]
Use the notation of Sections~2 and~3.  We first show that the case (i) of Theorem~\ref{thm:symplectic-selection} is impossible.  Choose arbitrary $x_i\in T_i$.  By Lemma~\ref{lem:central-products} and Lemma~\ref{lem:semidihedral}, choose $0\ne w_i\in W$ so that the involution $t_{w_i}$ has image $x_i$.  The fixed-space cover \eqref{eq:cover} gives $y\in h^H$ fixing $(w_1,\ldots,w_n)\in W^n$.

The element $y$ belongs to $P_0$.  Let $\bar y$ denote $y$ in the cases $|W|=25$ and $|W|=81$, and $yK$ in the case $|W|=q^2$.  Thus $\bar y\in P$ and $\pi(\bar y)\in\mathcal Y$.  If the $i$th component of $y$ is nonidentity, it lies in $\Fitt(J)$ and fixes $w_i$; by \eqref{eq:unique-fixer} it equals $t_{w_i}$ and therefore has image $x_i$ in $E_i$.  Consequently $(x_i)\in\mathcal C_{\pi(\bar y)}$.  Hence the sets $\mathcal C_y$, $y\in\mathcal Y$, cover $\prod_iT_i$, and the case (i) cannot occur.  Theorem~\ref{thm:symplectic-selection} therefore gives
\begin{equation}
  \mathcal Y\cap\rad_B(\mathcal U)=\varnothing.
  \label{eq:off-radical}
\end{equation}

Put $D=P'$ and let $\beta=\beta_0|_D$.  By Lemma~\ref{lem:commutator-data}, this is a $\bar H$-invariant linear character of $D$.  For each $\bar y\in\bar h^{\bar H}$, \eqref{eq:off-radical} gives $u\in\mathcal U$ with $B(\pi(\bar y),u)=1$.  Since $\pi(P)=\mathcal U$, choose $p\in P$ with $\pi(p)=u$.  Formula \eqref{eq:commutator-formula} gives
\[
  \beta([\bar y,p])=-1.
\]
Lemma~\ref{lem:central-commutator-zero}, applied to
$P\normal\bar H$, yields $\bar\psi\in\Irr(\bar H)$ with
$\bar\psi(\bar h)=0$.  In the cases $|W|=25$ and $|W|=81$, this is already a character of $H$.  In the case $|W|=q^2$, inflate $\bar\psi$ from $H/K$ to $H$.  In either case we obtain $\psi\in\Irr(H)$ with $\psi(h)=0$, as required.
\end{proof}

Theorem~\ref{thm:main} follows from Proposition~\ref{prop:linear-reduction} and Theorem~\ref{thm:linear}.

\section*{Acknowledgments}

ChatGPT was consulted during the preparation of this work for general language editing and for exploratory discussions of mathematical ideas. All mathematical results and proofs in the final manuscript were independently verified by the authors.

\bigskip
\noindent\textsc{School of Mathematical Sciences, Peking University,
Beijing 100871, People's Republic of China}

\noindent\textit{Email address of Quanfu Yan:}
\href{mailto:qyan5@kent.edu}{qyan5@kent.edu}

\noindent\textit{Email address of Jiping Zhang:}
\href{mailto:jzhang@math.pku.edu.cn}{jzhang@math.pku.edu.cn}

\end{document}